\documentclass[11pt,a4paper]{amsart}
\usepackage[latin1]{inputenc}
\usepackage[english]{babel}
\usepackage[T1]{fontenc}
\usepackage{amsmath}
\usepackage{amssymb}
\usepackage{amsthm}
\usepackage{array,booktabs}
\usepackage{graphicx}
\usepackage{pagecolor}
\usepackage{hyperref}
\usepackage{euflag}
\usepackage{lmodern}
\usepackage{tcolorbox}
\usepackage{quoting}
\usepackage{enumerate}
\usepackage{mathtools}
\usepackage{csquotes}
\usepackage{stmaryrd}
\usepackage{cite}
\usepackage{tikz}
\usepackage{tikz-cd}
\usetikzlibrary{matrix,arrows,decorations.pathmorphing}
\usepackage{wrapfig}
\renewcommand{\Re}{\operatorname{Re}}

\def\C{\ensuremath\mathbb{C}}
\def\A{\ensuremath\mathbb{A}}

\def\Z{\ensuremath\mathbb{Z}}
\def\Q{\ensuremath\mathbb{Q}}

\def\N{\ensuremath\mathbb{N}}
\def\Hb{\ensuremath\mathbb{H}}
\def\F{\ensuremath\mathbb{F}}

\usepackage{hyperref}
\usepackage[final]{pdfpages}
\usepackage{verbatim}
\newtheorem{theorem}{Theorem}[section]
\newtheorem{definition}[theorem]{Definition}
\newtheorem{corollary}[theorem]{Corollary}
\newtheorem{lemma}[theorem]{Lemma}

\newtheorem{proposition}[theorem]{Proposition}

\theoremstyle{remark}
\usepackage{todonotes}
\newtheorem{remark}{Remark}[section]

\def\FF{\mathcal{F}}

\DeclareMathOperator{\Hom}{Hom}

\DeclareFontFamily{U}{wncy}{}
    \DeclareFontShape{U}{wncy}{m}{n}{<->wncyr10}{}
    \DeclareSymbolFont{mcy}{U}{wncy}{m}{n}
    \DeclareMathSymbol{\Sh}{\mathord}{mcy}{"58} 
\def\modulo{\text{ \rm mod }}

\def\pmod{\text{ \rm mod }}
\numberwithin{equation}{section}

\numberwithin{equation}{section}

\begin{document}
\author{Daniel Kriz}

\address{Universit\`{a} degli Studi di Milano, Dipartimento di Matematica ``Federigo Enriques'', Via Cesare Saldini 50, 
20133 Milan (MI), Italy}
\email{\href{mailto:daniel.kriz@unimi.it}{daniel.kriz@unimi.it}}

\urladdr{\href{https://sites.google.com/view/dkriz/home}{https://sites.google.com/view/dkriz/home}}

\author{Asbj\o{}rn Christian Nordentoft}

\address{Universit\'{e} Paris-Saclay, Laboratoire de Math\'{e}matiques d'Orsay, 307, rue Michel Magat, 91405 Orsay Cedex, France}

\email{\href{mailto:acnordentoft@outlook.com}{acnordentoft@outlook.com}}

\urladdr{\href{https://sites.google.com/view/asbjornnordentoft/}{https://sites.google.com/view/asbjornnordentoft/}} 

\thanks{The first author was supported by the Simons Collaboration on Perfection in Algebra, Geometry, and Topology, award ID MP-SCMPS-00001529-10. The second author was supported by the Fondation Math\'{e}matique Jacques Hadamard.}

\title{$p$-adic moments of $L$-functions}
\begin{abstract}
We obtain a formula for the $p$-adic valuation of weighted moments of central $L$-values of holomorphic cusp forms twisted by Dirichlet characters of order $p$. In some cases we give an arithmetic interpretation of the constants in the formula. The result is obtained via the study of the {\it digit map}, turning a horizontal $p$-adic measure into a vertical one, applied to the horizontal $p$-adic $L$-functions as defined by the authors in \cite{KrizNordentoft}. 
\end{abstract}
\maketitle
\section{Introduction}
The archimedean evaluation of moments of families of $L$-functions has a long and fruitful history starting with the classical result of Hardy--Littlewood for the Riemann $\zeta$-function,
\begin{equation}
    \int_{-T}^T \left|\zeta\left(\frac{1}{2}+it\right)\right|^2 dt\sim T\log T,\quad T\rightarrow \infty, 
\end{equation}
see \cite[Ch. VII]{Titchmarsh86}. More generally, let $\mathcal{F}$ be a family of (automorphic) $L$-functions and for $X\geq 1$ put 
$$\mathcal{F}(X):=\{\pi\in \mathcal{F}: \mathfrak{c}(\pi)\leq X\},$$ 
where $\mathfrak{c}(\pi)$ denotes the analytic conductor \cite[Sec.\ 5]{IwKo}. The theory of (archimedean) moments of $L$-functions is concerned with obtaining the asymptotic valuation of the averages:
$$\frac{1}{|\FF(X)|}\sum_{\pi\in\FF(X)} L(1/2,\pi),$$
as $X\rightarrow \infty$. Here $L(s,\pi)$ denotes the automorphic $L$-function associated to an automorphic form $\pi$, which in the introduction we conveniently normalize so that $s=1/2$ is the central value. There exist conjectural asymptotic formulas for such moments with connections to Random Matrix Theory \cite{CFKRS05} predicting a main term of the shape $c_{k,\FF}(\log X)^{n_\FF}$ for some $c_{k,\FF}\in \C$ and $n_\FF\in \Z_{\geq 0}$ depending on the family $\FF$. Such moments have been determined using a variety of methods. As was pioneered by Selberg \cite{Selberg43}, the theory is very robust in the sense that one can often also determine 
$$\frac{1}{|\FF(X)|}\sum_{\pi\in\mathcal{F}(X)} L(1/2,\pi)r(\pi),$$
with $r(\pi)$  a (short) Dirichlet polynomial. Choosing the weights $r(\pi)$ appropriately (mollifier, amplifier, resonator) leads to applications to non-vanishing, subconvexity and large values, see e.g. the introduction of \cite{BlFoKoMiMiSa18}. The standard approach to obtaining asymptotic evaluations of the above moments goes via expressing the $L$-function as a convergent smooth sum using the \emph{approximate functional equation}, interchanging the sums, then using appropriate trace and summation formulas. In many cases this leads to deep and difficult questions in exponential sums which, in general, have yet to be solved.

In the case where $\pi$ is motivic (meaning here that $L(s,\pi)$ is a motivic $L$-function, automorphically normalized) and $s=1/2$ is a \emph{critical value} of $L(s,\pi)$ in the sense of  Deligne \cite{Deligne79}, then Deligne's conjecture on critical values predicts that there should be a naturally defined period $\Omega_\pi\in \C^\times$ such that 
$$ L(1/2,\pi)/\Omega_\pi\in \overline{\Q},$$
is an algebraic number. One can then consider  a family $\mathcal{F}$ of automorphic forms passing through $\pi$ and in many cases one can pick the periods $\Omega_\pi=\Omega_\mathcal{F}$ to only depend on the family $\mathcal{F}$. In any case, one may pick any embedding $\overline{\Q}\hookrightarrow \overline{\Q}_p$ into a fixed algebraic closure of the $p$-adic numbers and consider the images
$$L(1/2,\pi)/\Omega_\pi\in \overline{\Q}\hookrightarrow \overline{\Q}_p.$$
A natural question then arises: is there a theory of $p$-adic moments of $L$-functions, i.e. moments constructed from the above images and appropriate $p$-adic analogues of the weights $r(\pi)$? 
\subsection{Previous examples of $p$-adic moments}Examples of known results on $p$-adic valuations of (averages of) $L$-functions are rare. Unsurprisingly, often one has to order the $L$-values associated to $\mathcal{F}$ in a way corresponding to the $p$-adic size of the analytic conductor rather than by the archimedean size of the conductor. For $f$ a holomorphic cusp form of even weight $k$ and $\chi$ a Dirichlet character,  we define the $L$-function as the analytic continuation of
\begin{equation}\label{eq:Lfunction}L(f,\chi,s):=\sum_{n\geq 1} \frac{a_f(n)\chi(n)}{n^s},\quad \Re s>\frac{k+1}{2},\end{equation}   
where  $f(z)=\sum_{n\geq 1} a_f(n)q^n$ with $q=e^{2\pi i z}$. In this case the central value $s=k/2$ is a critical value\footnote{Note that  the central value is $s=k/2$ (instead of $s=1/2$) where $k$ is the weight. To emphasize this difference in normalization we will write the $s$-parameter last.  }. 
\subsubsection{Cyclotomic twist families} The first families $\mathcal{F}$ for which asymptotics on $p$-adic valuations of $L$-values were obtained were those arising from cyclotomic twist families on the cyclotomic $\Z_p$-tower over $\Q$ (which we henceforth refer to as ``$p$-cyclotomic twists''). The most fruitful method for obtaining such results is Iwasawa theory. For $\mathcal{F}$ equal to a family of $p$-cyclotomic twists of a fixed Dirichlet character, results of this type go back to the work of Iwasawa \cite{Iwasawa}. The first results on non-abelian $\mathcal{F}$ concerned the case of $p$-cyclotomic twists of the automorphic representation attached to an elliptic curve $E/\Q$ with good and ordinary reduction at $p$ (meaning that $p+1-a_E(p)=\#E(\F_p)\not\equiv 1\modulo p $). Let $f_E$ denote the weight $2$ cusp form associated to $E$ via modularity and let $v_p : \C_p \rightarrow \Q \cup \{\infty\}$ denote the standard $p$-adic valuation (normalized by $v_p(p) = 1$). Then by the work of Mazur-Swinnerton-Dyer \cite{MazurSD} there exist constants\footnote{Note that with our normalization $\mu$ is not necessarily equal to the $\mu$-invariant of the elliptic curve $E$. In particular, $\mu$ might be negative e.g. if $E$ has rank zero and $p$ divides the denominator of the BSD-formula.} $\mu\in \Z,\lambda\in \Z_{\geq 0}$ (which are finite, by the work of Rohrlich on nonvanishing $p$-cyclotomic twists \cite{Rohrlich84}) such that for $n$ sufficiently large and for an even Dirichlet character $\chi \mod p^n$ of $p$-power order it holds that 
$$v_p\left(\frac{L(f_E,\chi,1)}{\Omega_E^+}\right)=\mu+\frac{\lambda}{p^n-p^{n-1}},$$
where $\Omega_E^+$ denotes the (real) N\'{e}ron period of $E$, which is the period associated by Deligne in this setting. Here we get an asymptotic formula for the $p$-adic valuation of a \emph{single} $L$-function. Such regular behavior is never expected to hold in the archimedean aspect. We note that a similar formula holds for many such $p$-cyclotomic twist families, including higher rank examples \cite[Theorem 13.13]{Washington}, \cite{DiJaRa20}.  

\subsubsection{Stabilization of the trace formula}The second example comes from stabilization of trace formulas as observed by Michel and Ramakrishnan \cite{MichelRam}. Here we will state a specific case of the theorem suited to our perspective, see also \cite{Nelson13}, \cite{FeigonWhitehouse09} for generalizations. Let $p>3$ be prime. Let $-D<0$ be a fundamental negative discriminant with corresponding class number $h_{-D}$ and quadratic character $\chi_{-D}$. Let $N$ be prime and denote by $\FF_2(N)$ the Hecke eigenbasis of weight $2$ cuspidal holomorphic forms of level $N$. Then for $f\in \FF_2(N)$ we can find a period $\Omega_f\in \C^\times$ such that for any fundamental discriminant $-D<0$  
\begin{equation}
    \frac{L(f,\chi_{-D},1)L(f,1)}{\Omega_f}\in \overline{\Q}.
\end{equation}
Here more precisely we pick $\Omega_f=(N-1)^{-1}\langle f,f\rangle (2\pi)^2$ where $\langle f,f\rangle=\int_{\Gamma_0(N)\backslash \Hb} |f(z)|^2 dxdy$ denotes the Petersson norm, see \cite[Sec.\ 5]{MichelRam}.
It follows from \cite[p.\ 5]{MichelRam} that for fixed $-D<0$ and $n_0$ sufficiently large  (depending on $D$) and $N\equiv 1\modulo p^{n_0}$ it holds that 
\begin{equation}\label{eq:stabiliz}
    v_p\left(  \sum_{f\in \FF_2(N)} \frac{L(f,\chi_{-D},1)L(f,1)}{\Omega_f} \right)= v_p(h_{-D})-\frac{1}{2}v_p(D).
\end{equation}
In this case we see that the $p$-adic valuation stabilizes completely from a certain point. Note that the condition on $N$ can be expressed as $v_p(\varphi(N))=v_p(N-1)\geq n_0$ where $\varphi$ denotes  Euler's phi function. 
\subsection{Horizontal $p$-adic moments} In this note we add a new class of examples to this list; namely families obtained by twisting by characters of order \emph{exactly} $p$. We will now describe an imprecise version of our main result Theorem \ref{thm:widemoment2} in the case of an elliptic curve $E/\Q$: Consider a sequence of square-free integers $L_1,L_2,\ldots$ such that $L_n|L_{n+1}$ and all prime divisors $\ell|L_n$ satisfy $\ell\equiv 1\modulo p$ and $a_E(\ell)\not\equiv 2\modulo p$, meaning that $\ell$ is a \emph{Taylor--Wiles prime} as defined in Definition \ref{def:TW}. When $E$ is non-CM the existence of such a sequence is automatic for $p\geq 13$ \cite{Zywina15}. Let $m\geq 1$ be an integer. Then there exist constants $ \mu\in \Z\cup \{\infty\}$  and $\lambda\in \Z_{\geq 0}$ such that for $n$ sufficiently large  
\begin{equation}\label{eq:padicmoment}
   v_p\left( \frac{1}{p^n}\sum_{\substack{\chi\modulo L_n\\ \text{ of order $p$}}} 
  \left(\frac{L^\ast(f_E,\chi,1)}{\Omega_E^+}\right)^m u_n(\chi)\right)=\mu-\frac{1}{p-1}+\frac{\lambda+1}{p^n-p^{n-1}},
\end{equation}
where $u_n(\chi)$ denotes a specific product of cyclotomic units (\ref{eq:cycunit}) generating increasingly, in $n$, ramified extensions of $\Q_p$ and $L^\ast(f_E,\chi,1)$ denotes the twisted $L$-function (\ref{eq:Lfunction}) with modified Euler factors at primes dividing $L_n$. Notice that by the above assumptions it holds that $v_p(\varphi(L_n))\rightarrow \infty$ as $n\rightarrow \infty$, resembling  the limit considered in (\ref{eq:stabiliz}). Conjecturally \cite{DaFeKi07} for $p\geq 7$  it should always hold that $ \mu<\infty$ and this can be proved if $L(f_E,1)\neq 0$, see Remark \ref{rem:Kurihara}. Determining the invariants  $\mu,\lambda$ in general is an interesting and challenging problem. In low rank cases it can be achieved via a conjecture of Kurihara as discussed in Section \ref{sec:constant}.      

\begin{remark}
Since we are considering twists by characters of \emph{fixed} order $p$, the $L$-values $L(f_E,\chi,1)/\Omega_E^+$ all belong to the field $\Q(\zeta_p)$ and so without the highly ramified weights $u_n(\chi)$ the $p$-adic valuations of the moments considered above, most likely, behave very erratically. In this sense, the highly ramified weights have the effect of smoothing out the sum. Note that there is a huge amount of $p$-adic cancellation in the sum to compensate for the term $p^{-n}$ meaning that the  $L$-values are highly correlated with cyclotomic units.  Interestingly,  there is quite a bit of liberty in choosing the $p$-adic weights for which such a formula holds, see Remark \ref{rem:moregeneraldigit}, in analogy with the weight $r(\pi)$ in  the achimedean case.\end{remark} 
     
\subsubsection{Statement of results}
 The approach in this paper is to construct a power series interpolating the (twisted) moments of $L$-functions which reveals a general $p$-adic structure of such moments. This power series is constructed from the \emph{horizontal $p$-adic $L$-function} as defined in \cite[Definition 5.3]{KrizNordentoft} via the main novelty of this paper; the \emph{digit construction}. 

We will now describe a precise version of (\ref{eq:padicmoment}) for general holomorphic newforms of even weight. Let $m\geq 1$ and consider holomorphic newforms $f_1,\ldots ,f_m$ of respective levels $N_1,\ldots,N_m$, even weights $k_1,\ldots, k_m$ and trivial nebentypus. We say  that $p$ is \emph{$(f_1,\ldots, f_m)$-good} if there exists a place $\mathfrak{p}|p$ of the compositum of the Hecke fields of $f_1,\ldots, f_m$ such that there exist infinitely many \emph{joint Taylor--Wiles primes mod $\mathfrak{p}$} for $f_1,\ldots, f_m$, i.e.  primes $\ell$ such that 
\begin{equation}\label{eq:TWcondjoint} \ell\equiv 1\modulo p,\quad (\ell,N_i)=1,\quad a_{f_i}(\ell)\not\equiv 2\modulo \mathfrak{p},\quad i=1,\ldots,m,\end{equation} 
where $a_{f_i}(\ell)$ denotes the $\ell^\mathrm{th}$ Fourier coefficient of $f_i$ (see Definition \ref{def:TW}). Denote by $\C_p$ the complex $p$-adic numbers and by $v_p:\C_p\rightarrow \Q\cup\{\infty\}$ the unique valuation with $v_p(p)=1$. We fix an embedding $\overline{\Q}\hookrightarrow  \C_p$ corresponding to a prime of $\overline{\Q}$ above $\mathfrak{p}$ and a sequence of compatible, primitive $p$-power roots of unity $\zeta_{p}, \zeta_{p^2},\ldots\in \C_p$, meaning that $(\zeta_{p^{n_1}})^{p^{n_2}}=\zeta_{p^{n_1-n_2}}$ for all $n_1\leq n_2$. 

\begin{theorem}\label{thm:widemoment2}
Let $f_1,\ldots, f_m$ be holomorphic newforms as above (not necessarily distinct) and let $p$ be a $(f_1,\ldots, f_m)$-good prime.  Let $\ell_1,\ell_2,\ldots$ be a sequence of primes congruent to $1$ modulo $p$ such that for $n$ large enough $\ell_n$ is a joint Taylor--Wiles prime for $f_1,\ldots, f_m$. For $i\geq 1$ let $b_i\in (\Z/\ell_i)^\times $ be a choice of generator and put $L_{n}=\ell_1\cdots \ell_{n}$ for $n\geq 1$. 

Then there constants  $\mu\in \Q_{\geq 0}\cup\{ \infty\}$ and $ \lambda\in \Z_{\geq 0}$ such that for all $n$ large enough the following holds: let $\chi_n$ be the order $p$ Dirichlet character of conductor $\ell_n$ such that $\chi_n(b_n)=\zeta_p$. Then we have that
\begin{align}\label{eq:widemoment2}
v_p& \left(\frac{1}{p^{n-1}}\sum_{\substack{\chi\modulo L_{n-1}\\\mathrm{even},\,\chi^p=\mathbf{1}}} \left(\prod_{i=1}^m\frac{L^{\ast}(f_i,\chi\chi_n,k_i/2)}{\Omega_{f_i}^+}\right)u_{n-1}(\chi)\right)\\
\nonumber &=\mu-\frac{1}{p-1}+\frac{\lambda+1}{p^{n}-p^{n-1}},
\end{align}
where the sum is over even Dirichlet characters of order $p$ with conductor dividing $L_{n-1}$, $L^{\ast}(f_i,\chi,s)/\Omega_{f_i}^+$ denotes the algebraically normalized $L$-function of $f_i$ twisted by $\chi$ with modified Euler factors at primes dividing the conductor of $\chi$ (see (\ref{eq:interpolationfinal}) for the precise formula), and $u_{n-1}(\chi)$ is a product of cyclotomic units:
\begin{equation}\label{eq:cycunit}u_{n-1}(\chi):=\prod_{i=1}^{n-1} \frac{\zeta_{p^{n-i+1}}-1}{\zeta_{p^{n-i+1}}\overline{\chi_i(b_i)}-1},\end{equation}
where $\chi_i$ denotes the restriction of $\chi$ to $(\Z/\ell_i)^\times$.

Furthermore, we have $\mu=\infty$ (i.e. the quantity inside the argument of $v_p$ on the left-hand side of (\ref{eq:widemoment2}) is zero for $n$ sufficiently large) if and only if 
$$\prod_{i=1}^m L(f_i,\chi,k_i/2)=0,$$ 
for all Dirichlet characters $\chi$ with $\chi^p=1$ and conductor given by a product of the primes $\ell_1,\ell_2,\ldots$ (including the trivial character).
\end{theorem}

\begin{remark} Theorem \ref{thm:widemoment2} follows from a general result Proposition \ref{prop:widemoment} for horizontal measures. When comparing the two, note that the quantity $\mu$ appearing in (\ref{eq:widemoment2}) is equal to the quantity $\tfrac{\mu}{e}$ appearing in (\ref{eq:widemoment}) with $e$ equal to the ramification index of $F_{\mathfrak{p}}/\Q_p$ where $F$ is the compositum (over $\Q$) of the Hecke fields of $\{f_i\}_{i = 1}^m$ and $F_{\mathfrak{p}}$ is its completion at the prime $\mathfrak{p}|p$ chosen above. In other words, if $\pi$ is a uniformizer of $F_{\mathfrak{p}}$ then $v_p(\pi) = \tfrac{1}{e}$.
\end{remark}

\begin{remark}\label{rem:Kurihara}
We expect that it always holds that $\mu<\infty$ (at least if the Taylor--Wiles primes are not too sparse). The dependence of the invariants $\mu,\lambda$ as in (\ref{eq:widemoment2}) on the modular forms and the Taylor--Wiles primes (as well as the choice of primitive roots $b_n\in (\Z/\ell_n)^\times$) seems very mysterious in general. However, in the case where the modular forms correspond to elliptic curves $E_1,\ldots, E_m$, then in certain low rank cases $\mu,\lambda$ should be related to key arithmetic invariants:
\begin{enumerate}
\item firstly, if $L(f_{E_i},1)\neq0,i=1,\ldots, m$ and $p$ does not divide the BSD-formula for any of the elliptic curves $E_1,\ldots, E_m$ then $\mu=0$ and $\lambda=0$, 
\item secondly, if $m=1$ and $E_1$ satisfies the mod $p$ Kurihara Conjecture with $r=1$ and $\ell_1=q_1$, meaning that (\!\cite[Equation (1.6)]{KrizNordentoft}) is satisfied:
\begin{equation}\label{eq:Kolderiv}\sum_{a=1}^{q_1}a\left\langle \tfrac{(\zeta_{q_1})^{a}}{q_1} \right\rangle^+_E\not\equiv 0\modulo p,\end{equation}
where $\langle x\rangle_E^+= \Re2\pi i \int_x^\infty f_E(z) dz/\Omega_E^+ $ denotes \emph{modular symbols} and 
$\zeta_{q_1}$ denotes a generator of $(\Z/q_1)^\times$, then we have $\mu=0$ and $\lambda=1$. The mod $p$ Kurihara conjecture has been proved in many cases \cite{BurungaleCastellaGrossiSkinner}.
\end{enumerate}
We refer to Section \ref{sec:constant} for more details.
\end{remark}
\begin{remark}
    In \cite{KrizNordentoft} the authors obtained information about the $p$-adic valuation of the individual $L$-values in the same family as that of Theorem \ref{thm:widemoment2} (see Theorem 1.6 and Corollary 5.20 in \emph{loc.\ cit.}). But nothing was said about the $p$-adic oscillations, which conjecturally (for elliptic curves by the BSD conjectures) should encode deep arithmetic information. Clearly the valuation of the above moments depends crucially on these oscillations and how they interact with cyclotomic units, which themselves carry arithmetic meaning. 
\end{remark}
\begin{remark}
The above $p$-adic valuation arises from $p$-adic properties of the underlying automorphic periods which force a rigid structure on the moments. Different incarnations of automorphic periods implying (surprising) structure in moments of $L$-functions in the archimedean aspect include the topics of \emph{spectral reciprocity} \cite{BlomerKhan19}, \cite{BlomerLiMiller19} and \emph{wide moments} \cite{NordentoftWide2}, \cite{NordentoftWide1}. 
\end{remark}

\section{Recollections on $p$-adic measures}\label{sec:recoll}
In this section we will recall some general facts about \emph{horizontal} and \emph{vertical measures}.  
\subsection{The Amice transform and the Weierstrass Preparation Theorem}In this section we recall some classical results on vertical Iwasawa algebras over $p$-adic rings. For simplicity, we will restrict ourselves to our setting where $R \subset \C_p$ is a $p$-adically complete subring. Let $R\llbracket \Z_p\rrbracket := \varprojlim_nR[\Z/p^n]$, where the inverse limit is given by the projections $\Z/p^n \rightarrow \Z/p^{n-1}$. Thus $R\llbracket \Z_p\rrbracket$ is the space of $R$-valued measures on $\Z_p$, i.e. the $R$-linear topological dual of the space of continuous functions $\Z_p \rightarrow \Z_p$. Recall Mahler's theorem \cite[p. 52]{Washington}, which says that for any continuous function $f : \Z_p \rightarrow \Z_p$ we can write $f = \sum_{n = 0}^{\infty}a_n\binom{x}{n}$ for uniquely determined $a_n \in \Z_p$ with $\lim_{n \rightarrow \infty}v_p(a_n) = \infty$, where $\binom{x}{n} \in \Q[x]$ is the usual binomial coefficient (which is easily seen to take values in $\Z_p$ for $x \in \Z_p$, as it is $p$-adically continuous and takes values in $\Z$ for $x \in \Z$). In particular, we see that any $\nu \in R\llbracket \Z_p\rrbracket$ is uniquely determined by its values $\nu(\binom{x}{n}) \in R$. Precisely, we have the following  result of Amice (which is sometimes referred to as the \emph{Amice transform}, \cite[Theorem 7.1, Chapter 12]{Washington}): 

\begin{theorem}\label{thm:Amice}There is a canonical isomorphism of $R$-algebras
\begin{equation}\label{eq:amice}R\llbracket \Z_p\rrbracket \cong R\llbracket T\rrbracket, \hspace{1cm} \nu \mapsto \sum_{n = 0}^{\infty}\nu\left(\binom{x}{n}\right)T^n.\end{equation}
\end{theorem}

We also recall the classical \emph{Weierstrass preparation theorem} (\!\cite[Theorem 7.3]{Washington}), adapted to our setting.

\begin{theorem}\label{thm:Weierstrassprep}Suppose $R \subset \C_p$ is the ring of integers of a finite extension of $\Q_p$. Let $\pi$ be a uniformizer of $R$. Then any non-zero element $f \in R\llbracket T\rrbracket$ admits a unique factorization $f = \pi^{\mu(f)}g(T)u(T)$ where $\mu(f) \in \Z_{\ge 0}$, $g(T) = T^{\lambda(f)} + a_{\lambda(f)-1}T^{\lambda(f)-1} + \cdots + a_0$ with $a_i \in (\pi)$ for all $0 \le i \le \lambda(f)-1$, and $u(T) \in R\llbracket T\rrbracket^{\times}$.
\end{theorem}
In the discussion below we continue to denote the invariants associated to $f\in R\llbracket T\rrbracket$ as in the above theorem by $\mu(f),\lambda(f)$ and refer to them as, respectively the $\mu$-\emph{invariant} and $\lambda$-\emph{invariant of $f$}. For $f=0$ we set $\mu(f)=\infty$ and $\lambda(f)=0$.
\subsection{The horizontal Iwasawa algebra}In this section we recall some of the background material from \cite[Section 2.1]{KrizNordentoft}. Let $\N = \Z_{> 0}$. Suppose we are given a finite sequence of finite abelian groups $G_1, G_2,\ldots$. The product group $\prod_{n \in \N}G_n$ is canonically a projective limit
$$G_{\N} = \varprojlim_{\substack{A \subset \N\\\text{ finite}}}\prod_{n \in A}G_n$$
where the transition maps are given by the canonical projections
$$\pi_{A,A'} : \prod_{n \in A}G_n\twoheadrightarrow \prod_{n \in A'}G_n$$
if $A' \subset A$. For each finite set, we give $\prod_{n \in A}G_n$ the discrete topology. Then $G_{\N}$ has the induced product topology, which is the coarsest topology such that each projection $G_{\N} \rightarrow \prod_{n\in A}G_n$ is continuous. This also realizes $G_{\N}$ as a profinite group. Let $R$ be a commutative ring and define 
\begin{equation}\label{eq:defhoralg}\Lambda_{G_\N, R}^{\mathrm{hor}}:=R\llbracket G_\N \rrbracket= \varprojlim_{\substack{\pi_{A,A'} \\ A' \subset A \subset \N\\A, A'\text{finite}}}R\left[\prod_{n \in A}G_n\right].\end{equation}
We endow $\Lambda_{G_\N, R}^{\mathrm{hor}}$ with the inverse limit topology and refer to it as the \emph{horizontal Iwasawa algebra (associated to $G_\N$ and $R$)}. We note that the rings $\Lambda_{G_\N,R}^{\mathrm{hor}}$ are quite badly behaved: they are non-noetherian and not integral domains. For example, observe that for each $n\in \N$ we have canonical embeddings
\begin{align}\label{eq:nonintdom}R[G_n]\hookrightarrow \Lambda_R^{\mathrm{hor}},\end{align}
induced from the group embeddings $G_n\hookrightarrow \prod_{i=1}^{n'}G_i$ for $n'\geq n$ and clearly for $g\in G_n$ non-trivial of order $m\geq 1$ we have 
$$([g]-1)([g^{m-1}]+\ldots+[g]+1)=0\in R[G_n].$$ 
See \cite[Section 2.2]{KrizNordentoft} for more details.  


\subsection{Recollections on horizontal $p$-adic $L$-functions of holomorphic newforms}In this section we briefly recall some background results on the horizontal $p$-adic $L$-function attached to a holomorphic newform from \cite[Section 5]{KrizNordentoft}. We refer to \emph{loc.\ cit.}\ for full details. Throughout this section we let $f$ be a holomorphic newform of level $N$, even weight $k$, nebentype $\epsilon_f$ and Fourier expansion  $f(z)=\sum_{n\geq 1} a_f(n)q^n$ at infinity. Recall the definition of the $L$-function as in (\ref{eq:Lfunction}). Let $\mathcal{O}_f$ denote the ring of integers of the Hecke field $K_f$ of $f$. Given a prime $\mathfrak{p}$ of $\mathcal{O}_f$ we let $\mathcal{O}_{f,\mathfrak{p}}$ denote the $\mathfrak{p}$-adic completion of $\mathcal{O}_f$. First we recall the notion of \emph{Taylor--Wiles primes}.

\begin{definition}[Taylor--Wiles primes]\label{def:TW}
We define the set of \emph{Taylor--Wiles primes for $f$ (modulo $\mathfrak{p}$)} as: 
\begin{equation}\label{eq:TWdef}
TW(f;\mathfrak{p}):=\{\ell\text{ prime}: \ell\nmid N, p|\ell-1, a_f(\ell)-1-\epsilon_f(\ell)\in (\mathcal{O}_{f,\mathfrak{p}})^\times\}, 
\end{equation}
We say that $\mathfrak{p}$ is \emph{$f$-good} if $TW(f;\mathfrak{p})$ has positive density among all primes.
\end{definition}
The condition of $\mathfrak{p}|p$ being $f$-good follows from a \emph{big image} assumption and is thus generically true for non-CM forms. We refer to \cite[Section 4]{KrizNordentoft} for more details.
  
In \cite{KrizNordentoft} the authors showed how to $p$-adically interpolate the $L$-values given by $p$-power order characters with conductor given by a product of Taylor--Wiles primes. 

\begin{definition}[Horizontal $p$-adic $L$-function of a holomorphic newform, Definition 5.3 of \cite{KrizNordentoft}]\label{def:padicL}
   Let $p$ be a prime number and let $\mathfrak{p}|p$ be a place of $K_f$ above $p$ which is $f$-good in the sense of Definition \ref{def:TW}. Let $R$ denote the ring of integers of the completion $K_{f,\mathfrak{p}}$ of $K_f$ at $\mathfrak{p}$. Let $r\geq 0$ be an integer, $\pm$ a sign, and let $\mathcal{L}=(\ell_n)_{n\in \N}$ be a sequence of distinct primes not dividing $N$, congruent to $1$ modulo $p$, and such that $\ell_n\in TW(f;\mathfrak{p})$ for $n\geq r+1$. For $n\in \N$ put   $m_{n}=v_p(\ell_n-1)$.
   
   Let 
\begin{equation}\label{eq:horpadicL}\nu^\pm_{f,\mathcal{L},r}\in  R\left\llbracket \prod_{n\in \N} \Z/p^{m_n} \right\rrbracket  ,
\end{equation}
  be the \emph{horizontal $p$-adic $L$-function of $f$ associated to $(\mathcal{L},r,\pm)$} from \cite[Definition 5.3]{KrizNordentoft}. 
\end{definition}
We emphasize that the fact that we can allow for finitely many non-Taylor--Wiles primes $\{\ell_1,\ldots,\ell_r\}$, suppressed from the notation, is key for many applications, see Remark \ref{rem:Kurihara}.

The $p$-adic measure $\nu^{\pm}_{f,\mathcal{L},r}$ is uniquely determined by the interpolation property \cite[Corollary 5.4]{KrizNordentoft} which we will now recall.  Suppose we are given a finite order Dirichlet character $\chi$ of  conductor $D_{\chi}$ dividing $\prod_{n \in \N}\ell_n$ and sign $\pm$, i.e. $\chi(-1) = \pm 1$. Then as in \cite[Equation (5.13)]{KrizNordentoft} we define the modified $L$-value of $f$ twisted by $\chi$:
\begin{align}\nonumber L^\ast_f(\chi)&:=\left(\prod_{1\leq i\leq r: \ell_i\nmid D_\chi}(a_f(\ell_i)\ell_i^{-(k-2)/2}-\chi(\ell_i)-\epsilon_f(\ell_i)\overline{\chi}(\ell_i))\right)\\
\label{eq:interpolationfinal} &\cdot \left(\prod_{i\geq r+1: \ell_i|D_\chi}(a_f(\ell_i)\ell_i^{-(k-2)/2}-\chi^{(i)}(\ell_i)-\epsilon_f(\ell_i)\overline{\chi}^{(i)}(\ell_i))\right)^{-1} \\
&\nonumber\cdot \tau(\overline{\chi})L(f,\chi,k/2)/\Omega_f^\pm.\end{align} 
Here $\chi^{(i)} \pmod{\frac{D_{\chi}}{\ell_i}}$ denotes the Dirichlet character obtained by restricting $\chi$ modulo  $\frac{D_{\chi}}{\ell_i}$, the twisted $L$-function $L(f,\chi,s)$ is defined as in (\ref{eq:Lfunction}), $\tau(\chi)$ is a Gauss sum, and $\Omega_f^{\pm}$ denotes a plus/minus period of $f$, see \cite[Section 3]{KrizNordentoft} for details. In particular, for $f=f_E$ corresponding to an elliptic curve $E/\Q$ one can pick $\Omega_{f_E}^\pm\in \Omega_E^\pm\Q^\times$ where $\Omega_E^\pm$ denotes the N\'{e}ron periods of $E$.

\begin{corollary}[Interpolation property of the horizontal $p$-adic $L$-function]\label{cor:measure}
Let $f$ be a newform of even weight $k$, level $N$ and nebentypus $\epsilon_f$. Let $\nu^\pm_{f,\mathcal{L},r}\in R\left\llbracket \prod_{n\in \N} \Z/p^{m_n} \right\rrbracket $ be a horizontal $p$-adic $L$-function of $f$ as in Definition \ref{def:padicL}. Let $\chi$ be a Dirichlet character of $p$-power order with conductor dividing $\prod_{n\in \N}\ell_n$. Write $\chi=\tilde{\chi}\circ \rho_\N$ in terms of the projection $\rho_\N: \prod_{n\in \N}(\Z/\ell_n)^\times \twoheadrightarrow \prod_{n\in \N}\Z/p^{m_n}$ and a character  $\tilde{\chi}:\prod_{n\in \N}\Z/p^{m_n}\rightarrow \C_p^\times$. Let $\pm$ denote the sign of $\chi$, i.e. $\chi(-1)=\pm 1$. Then we have \begin{equation}\label{eq:interpolationmeasure}\nu^\pm_{f,\mathcal{L},r}(\tilde{\chi})=L^\ast_f(\chi).\end{equation} 
\end{corollary}

\section{The digit construction}\label{sec:digitconstr}
In this section we introduce a method for turning horizontal measures as in \cite{KrizNordentoft} into usual \emph{vertical measures}, i.e. elements of the \emph{vertical Iwasawa algebra} $\Z_p\llbracket \Z_p\rrbracket$. Applying the Amice transform (Theorem \ref{thm:Amice}), we then produce a power series whose values can be related to the original horizontal measure via Fourier theory. Applying this construction to the \emph{horizontal $p$-adic $L$-functions} associated to holomorphic cusp forms constructed in \cite{KrizNordentoft} together with the Weierstrass Preparation Theorem (Theorem \ref{thm:Weierstrassprep}) then yields the $p$-adic valuation of the twisted moments of the shape (\ref{eq:padicmoment}).

The starting point for our construction is the \emph{set theoretic} bijection 
$$(\mathbb{Z}/p)^\N \cong_{\mathrm{Sets}} \mathbb{Z}_p,$$
given by the usual $p$-adic digit expansion (see Definition \ref{def:digitgen} for the general case). This yields an isomorphism of topological $\Z_p$-modules 
$$\Z_p\llbracket (\Z/p)^\N \rrbracket \cong \Z_p\llbracket \mathbb{Z}_p\rrbracket.$$
We call this latter map the \emph{digit map}. This allows us to obtain $p$-adic valuation of certain ``twisted moments'' of horizontal measures (Proposition \ref{prop:widemoment}). 

It is a general fact that given two pro-finite groups $G_1,G_2$ then any homeomorphism as profinite sets $G_1\xrightarrow{\sim} G_2$ induces an isomorphism of $R$-modules of the corresponding group rings $R\llbracket G_1 \rrbracket\xrightarrow{\sim}R\llbracket G_2 \rrbracket$ (which is \emph{not} necessarily an isomorphism of $R$-algebras). Of course, most such identifications will not reveal much about the properties of $R\llbracket G_1 \rrbracket$ as the ring structure is obscured. It turns out that the digit construction \emph{does} see parts of the group structure as we will see in the following sections (what we call the {\lq\lq}last digit trick{\rq\rq}). This ultimately leads to the appearance of cyclotomic units in Theorem \ref{thm:widemoment2}.  

\subsection{The digit map} As in the previous section let $G_1,G_2,\ldots $ be a sequence of finite abelian groups. Assume that for each $n \in \N$, $G_n$ admits $\mathbb{Z}/p^{m_n}$ as a quotient for some $m_n\geq 1$; for $k\geq 1$ we henceforth give $\mathbb{Z}/p^k$ the discrete topology, so that in particular any quotient map $G_n \rightarrow \mathbb{Z}/p^{m_n}$ is continuous. The purpose of this section is to show how to {\lq\lq}verticalize{\rq\rq} elements of $R\llbracket G_\N\rrbracket$ (i.e. horizontal measures), and turn them into elements of $R\llbracket \mathbb{Z}_p\rrbracket$  (i.e. vertical measures). This will be done via a certain ``digit map''. 

Given a sequence (i.e. ordered set) of positive integers $\underline{m}=(m_n)_{n\in \N}$ we define the \emph{digit group (associated to $\underline{m})$} by
$$\Pi_{\underline{m}}=\prod_{n\in \N}\mathbb{Z}/p^{m_n}.$$
For $n \in \N$ we have a natural projection
$$\pi_n' : \prod_{i=1}^n\mathbb{Z}/p^{m_i} \rightarrow \prod_{i=1}^{n-1}\mathbb{Z}/p^{m_i}$$
given by quotienting out by the $n^{\mathrm{th}}$ factor $\mathbb{Z}/p^{m_n}$. Then we have a canonical isomorphism of groups 
\begin{equation}\label{eq:aboveiso}\Pi_{\underline{m}} \cong \varprojlim_n \prod_{i=1}^n\mathbb{Z}/p^{m_i}.\end{equation}
Let $R$ be a commutative ring and define the \emph{digit algebra (associated to $\underline{m}$ and $R$)} as 
$$\Lambda_{R,\underline{m}}^{\mathrm{dig}} = \varprojlim_{\pi_n'} R\left[\prod_{i=1}^n\mathbb{Z}/p^{m_i}\right].$$
We note that despite the notation, $\Lambda_{R,\underline{m}}^{\mathrm{dig}}$ depends up to isomorphism on the underlying set of $\underline{m}$ (and not on the ordering of $\underline{m}$). 

Now pick projections $\rho_n:G_n\twoheadrightarrow \Z/p^{m_n}$ for $n\geq 1$ and consider the  associated continuous group surjection
\begin{equation}\label{eq:projectionrhogroup}\rho=\prod_{n\in \N} \rho_n :G_\N \twoheadrightarrow  \Pi_{\underline{m}}.
\end{equation} 
We then get a continuous $R$-algebra surjection (which we will denote by the same symbol) 
\begin{equation} \label{eq:projectionrho} \rho:\Lambda_R^{\mathrm{hor}} \twoheadrightarrow \Lambda_{R,\underline{m}}^{\mathrm{dig}}.\end{equation}
The following key definition allows us to verticalize the digit algebra to relate it to the usual Iwasawa algebra. 
\begin{definition}[The digit map]\label{def:digitgen}For $m\geq 1$ and $a \in \mathbb{Z}/p^m$, let $0 \le \tilde{a} \le p^m-1$ be the unique integer with $\tilde{a} \modulo p^m = a$. Let $\underline{m}=(m_n)_{n\in \N}$ be a sequence of positive integers. For any $n\in \N \cup\{\infty\}$ we have a commutative diagram of maps of sets:
\begin{equation}\label{eq:digitdefinition}\begin{tikzcd}[column sep = large]
\prod_{i=1}^n(\mathbb{Z}/p^{m_i})  \arrow{r}{d_{n,\underline{m}}} & \mathbb{Z}/p^{m_1+\ldots +m_n} \\
\Pi_{\underline{m}} \arrow{u}{} \arrow{r}{d_{\infty,\underline{m}}} & \mathbb{Z}_p \arrow{u}{}\\
  \end{tikzcd}.
  \end{equation}
where $d_{n,\underline{m}}$ is the set bijection given by 
$$\prod_{i=1}^n(\mathbb{Z}/p^{m_i}) \ni (a_i)_i \xmapsto{d_{n,\underline{m}}} \sum_{i = 1}^n\tilde{a}_ip^{m_1+\ldots+m_{i-1}} \in \mathbb{Z}/p^{m_1+\ldots+m_n}.$$
The commutativity of (\ref{eq:digitdefinition}) gives natural compatibilities of the diagrams (\ref{eq:digitdefinition}), which in turn gives the continuity of $d_{n,\underline{m}}$ for all $0 \le n \le \infty$. 
\end{definition}
We note that by the continuity of $d_{n,\underline{m}}$, for every $0 \le n < \infty$ we get an isomorphism of topological $R$-modules
$$d_{n,\underline{m}} : R\left[\prod_{i=1}^n\mathbb{Z}/p^{m_i}\right] \xrightarrow{\sim} R[\mathbb{Z}/p^{m_1+\ldots+m_n}].$$
This induces an isomorphism of topological $R$-modules
$$d_{\infty,\underline{m}} : \Lambda_{R,\underline{m}}^{\mathrm{dig}} \xrightarrow{\sim} R\llbracket \mathbb{Z}_p\rrbracket.$$
The case where there exists $m\geq 1$ such that $m_n=m$ for all $n\geq 1$ is of special interests to us as it will be connected to non-vanishing of twists of order \emph{precisely} $p^m$.  In this case we denote the associated digit group by  
$$(\Z/p^m)^\N=\varprojlim_{\pi'_n} \prod_{i=1}^n\mathbb{Z}/p^{m},$$
and denote the associated digit map by $d_{\infty,m}:(\Z/p^m)^\N\rightarrow \Z_p$ which we refer to as the \emph{$p^m$-adic digit map}. Similarly, we denote the associated digit algebra by
$$ \Lambda^\mathrm{dig}_{R,m}= \varprojlim_{\pi'_n} R\left[\prod_{i=1}^n\mathbb{Z}/p^{m}\right],$$ 
and refer to it as the \emph{$p^m$-adic digit algebra}. 
\begin{remark}\label{rem:moregeneraldigit}
One can consider more general ``digit maps'' obtained from any sequence of bijections $\prod_{i=1}^n\Z/p^{m_i}\xrightarrow{\sim} \Z/p^{\Sigma_{i=1}^n m_i}$ making the diagram (\ref{eq:digitdefinition}) commute. These correspond exactly to picking for each $n$ an automorphism of $ \Z/p^{m_n}$ as a \emph{set} rather than as a group). These maps also yield a verticalization procedure but with less favorable properties (see Remark \ref{rem:generaldigitfourier} below).   
\end{remark} 

\subsection{Fourier coefficients of the digit map}\label{sec:lastdigit}
In this section we will translate between characters of $\Z_p$ and those of $\prod_{n\in \N}\mathbb{Z}/p^{m_n}$ under the digit map. We start by recording a simple but key property  of the digit map: given an additive character of $\Z_p$ then the restriction to the {\lq\lq}last digit{\rq\rq} yields a character of the corresponding factor $\mathbb{Z}/p^{m_n}$.     
\begin{lemma}[The last digit trick]\label{lem:lastdigitgen}
Let $\psi:\Z_p\rightarrow \C_p^\times$ be an additive character of $\Z_p$ of conductor $p^n$. Then $\psi_{|p^{n-m}\Z_p}$ defines an order $p^m$ character of $p^{n-m}\Z_p/p^n\Z_p\cong \Z/p^m$.
\end{lemma}       
\begin{proof}
Since $\psi$ is trivial on $p^n\Z_p$ and has order $p^n$, we see that indeed $\psi^{p^{n-m}}$ has order $p^m$ and factors through the claimed quotient. This yields the assertion since $\psi$ is an additive character.
\end{proof}
Given a character $\psi:\Z_p\rightarrow \C_p^\times$ of conductor $p^n$ we will view it as a character $\psi:\Z/p^n\rightarrow \C_p^\times$ via the projection $\Z_p \rightarrow \Z/p^n$ and for $1\leq m<n$, we will denote by 
\begin{equation}\psi^{p^m}:\Z/p^{n-m}\rightarrow \C_p^\times,\end{equation} 
the character obtained by precomposing $\psi$ with the natural inclusion 
\[\Z/p^{n-m} \underset{\sim}{\xrightarrow{x \mapsto p^mx}}p^m\Z/p^n\Z \subset \Z/p^{n}\] given by multiplication by $p^m$. 

Let $G$ be a finite abelian group and denote by $\widehat{G}$ the group of $\C_p^\times$-valued characters. For $f:G\rightarrow \C_p$ and $\chi\in \widehat{G}$ we define the \emph{Fourier coefficient of $f$ against $\chi$} as:
\begin{equation}\langle f,\chi\rangle := \sum_{g\in G} f(g)\overline{\chi}(g). \end{equation} 
Then orthogonality of characters (i.e. Fourier inversion) yields the following Fourier expansion of $f$
\begin{equation}\label{eq:Finv}f(g)=\frac{1}{|G|}\sum_{\chi\in \widehat{G}}\langle f,\chi\rangle \chi(g),\quad g\in G. \end{equation}
The {\lq\lq}last digit trick{\rq\rq} above yields the following key restriction on the Fourier expansion of additive characters. This ultimately accounts for the appearance of cyclotomic units in Theorem \ref{thm:widemoment2}.  
\begin{lemma}\label{lem:lastdigitexpansion}
Let $\underline{m}=(m_n)_{n\in \N}$ be a sequence of positive integers and let  $\psi$ be an additive character of $\Z_p$ of conductor $p^{m_1+\ldots+m_{n-1}+k}$ with $1\leq k\leq m_n$. Then the Fourier coefficient of 
$$ \psi\circ (d_{n,\underline{m}})^{-1}: \prod_{i=1}^n (\Z/p^{m_i})\rightarrow \C_p^\times, $$
against a character $\chi=(\chi_1,\ldots, \chi_n)\in \prod_{i=1}^n \widehat{\Z/p^{m_i}}$ is zero unless $\chi_n=\psi^{p^{m_1+\ldots+m_{n-1}}}$ in which case we have
$$ \langle \psi\circ (d_{n,\underline{m}})^{-1},\chi\rangle =p^{m_n}\prod_{i=1}^{n-1}\frac{\psi(p^{m_1+\ldots+m_i})-1}{\psi(p^{m_1+\ldots+m_{i-1}})\overline{\chi_i}(1)-1}. $$ 
In particular, all non-zero Fourier coefficients have the same $p$-adic valuation.  
\end{lemma}
\begin{proof}
Let $\chi=(\chi_1,\ldots,\chi_n)\in \prod_{i=1}^n \widehat{\Z/p^{m_i}}$. Then since $\psi$ is an additive character, we see that for  $a=(a_1,\ldots,a_n)$ with $a_i\in \Z/p^{m_i}$; 
\begin{align} 
\psi((d_{n,\underline{m}})^{-1}(a))\overline{\chi(a)}&=\psi\left(\sum_{i=1}^{n}\tilde{a}_ip^{m_1+\ldots+m_{i-1}}\right)\overline{\chi}(a)\\
&= \prod_{i=1}^n (\psi(\tilde{a}_ip^{m_1+\ldots+m_{i-1}})\overline{\chi}_i(a_i)),
\end{align}
which yields the factorization
$$\langle \psi\circ (d_{n,\underline{m}})^{-1}, \chi \rangle= \prod_{i=1}^{n}\left(\sum_{a=0}^{p^{m_i}-1}\psi(ap^{m_1+\ldots+m_{i-1}}) \overline{\chi}_i(a)\right).$$
By Lemma \ref{lem:lastdigitgen} and orthogonality of characters, we see that the above vanishes unless 
$$\chi_n(1)=\psi^{p^{m_1+\ldots+m_{n-1}}}(1),$$
which yields the first claim. For $i<n$, we get a geometric series which amounts to 
\begin{align*}\sum_{a=0}^{p^{m_i}-1}\psi(ap^{m_1+\ldots+m_{i-1}}) \overline{\chi}_i(a)&= \sum_{a=0}^{p^{m_i}-1}(\psi(p^{m_1+\ldots+m_{i-1}})\overline{\chi}_i(1))^a\\
&=\frac{\psi(p^{m_1+\ldots+m_{i}})-1}{\psi(p^{m_1+\ldots+m_{i-1}})\overline{\chi_i}(1)-1}
\end{align*}
using that $\chi_i(1)^{p^{m_i}}=1$. This implies the desired formula. Finally it follows by the above formula that the $p$-adic valuation is the same for all the non-zero Fourier coefficients.
\end{proof}
\begin{remark}\label{rem:generaldigitfourier}
If one considers more general digit maps $d:\prod_{i=1}^\N \Z/p^{m_i}\xrightarrow{\sim} \Z_p$ as described in Remark \ref{rem:moregeneraldigit}, one can prove the following; with $\psi:\Z_p\rightarrow \C_p^\times$ and $\chi=(\chi_1,\ldots,\chi_n)$ as in Lemma \ref{lem:lastdigitexpansion}, we have that $\langle \psi\circ (d_{n,\underline{m}})^{-1}, \chi\rangle=0$ when $\chi_n$ is trivial. The Fourier coefficients do however not seem to admit a nice expression in this generality and in particular the $p$-valuation is not the same for the different Fourier coefficients. 
\end{remark}
From this we get the following translation between horizontal and vertical measures.
\begin{corollary}\label{cor:lastdigit}
Let $\underline{m}=(m_n)_{n\in \N}$ be a sequence of positive integers and let $\Lambda_{R,\underline{m}}^\mathrm{dig}$ be the associated digit algebra. Let $\nu\in \Lambda_{R,\underline{m}}^\mathrm{dig}$ be a horizontal measure and let $\psi$ be a character of $\Z_p$ of conductor $p^N$ with 
$$m_1+\ldots+m_{n}<N\leq m_1+\ldots+m_{n+1}.$$ 
Then we have
\begin{align}\label{eq:Fourierexpansionhor}
d_{\infty,\underline{m}}(\nu)(\psi)=&\frac{\psi(p^{m_1+\ldots+m_n})-1}{p^{m_1+\ldots+m_n}(\psi(1)-1)}\\
&\cdot\sum_{(\chi_i)\in \Pi_{i=1}^n\widehat{\Z/p^{m_i}}}c_{\psi}(\chi_1,\ldots, \chi_{n})\nu(\chi_1,\ldots,\chi_{n},\psi^{p^{m_1+\ldots+m_{n}}} ),   
\end{align} 
where 
\begin{align}c_{\psi}(\chi_1,\ldots, \chi_{n})=\prod_{i=1}^{n} \frac{\psi(p^{m_1+\ldots+m_{i-1}})-1}{\psi(p^{m_1+\ldots+m_{i-1}})\overline{\chi_i}(1)-1}. \end{align}
\end{corollary}
\begin{proof}
This follows directly from taking the Fourier expansion $d_{\infty,\underline{m}}(\nu)(\psi)$ in terms of the characters of $\prod_{i=1}^n \Z/p^{m_i}$ as in (\ref{eq:Finv}) combined with Lemma \ref{lem:lastdigitexpansion} by reordering the factors appropriately.
\end{proof}
\begin{remark}\label{rem:intermsof}
We will be applying Corollary \ref{cor:lastdigit} to measures $\rho(\nu)\in \Lambda_{R,\underline{m}}^\mathrm{dig}$ obtained from a horizontal measure $\nu\in \Lambda_R^\mathrm{hor}$ and a choice of projections $\rho_n:G_n\twoheadrightarrow \Z/p^{m_n}$. In this case the right-hand side of (\ref{eq:Fourierexpansionhor}) can be written in terms of $\nu$ evaluated at characters $\chi:G_\N\rightarrow  \C_p^\times$ which factor through the projection $\prod_{i=1}^n \rho_i$. In particular, if $m_n=1$ for all $n\geq 1$ this corresponds exactly to character of order $p$ (as well as the trivial character).  
\end{remark}

\subsection{Valuation of twisted moments}\label{sec:widemoment}
Using the above properties of the digit map we arrive at the following general result for twisted $p$-adic moments of horizontal measures. We let $\C_p$ be the $p$-adic complex numbers and fix a $p$-adic valuation $v_p:\C_p\rightarrow (-\infty,\infty]$ normalized so that $v_p(p)=1$.  
\begin{proposition}\label{prop:widemoment}
Let $R\subset \C_p$ be the ring of integers of a finite extension of $\Q_p$ with ramification index $e\geq 1$. Let $\nu\in \Lambda_R^\mathrm{hor}$ be a horizontal measure in the horizontal Iwasawa algebra defined from groups $(G_n)_{n\in \N}$. Let $m\geq 1$ be an integer and assume that there exists surjections $\rho_n:G_n\twoheadrightarrow \Z/p^{m},n\geq 1$. Let $\rho:\Lambda_{R}^\mathrm{hor}\rightarrow \Lambda_{R,m}^\mathrm{dig}$ be the pushforward to the $p^m$-adic digit algebra by the associated projection as in (\ref{eq:projectionrho}).
Let $\zeta_{p}, \zeta_{p^2},\ldots$ be a sequence of compatible, primitive $p$-power roots of unity, i.e. $(\zeta_{p^{n_1}})^{p^{n_2}}=\zeta_{p^{n_1-n_2}}$ for $n_1\leq n_2$. 

Let $ \mu\in \Z_{\geq 0}\cup\{ \infty\}$ and  $ \lambda\in \Z_{\geq 0}$  denote the $\mu$ and $\lambda$-invariants of the power series associated to $d_{\infty,m}(\rho(\nu))$ via the Amice transform. Then for $n$ large enough the following holds: If $\chi_0$ denotes the character of $\Z/p^m$ satisfying $\chi_0(1)=\zeta_{p^m}$, then we have
\begin{align}\label{eq:widemoment}
v_p&\nonumber\left(\frac{1}{p^{mn}}\sum_{(\chi_i)\in \widehat{(\Z/p^m)}^{n}} \left(\prod_{i=1}^{n}u_{m,n-i}(\chi_i)\right)\rho(\nu)(\chi_1,\ldots,\chi_{n},\chi_0 )\right)\\
&=\frac{\mu}{e}-\tfrac{1}{p^m-p^{m-1}} +\frac{\lambda+1}{p^{m(n+1)}-p^{m(n+1)-1}},
\end{align}
where $u_{m,n-i}(\chi_i)$ are cyclotomic units given by
$$u_{m,n-i}(\chi_i)= \frac{\zeta_{p^{m(n-i+2)}}-1}{\zeta_{p^{m(n-i+2)}}\overline{\chi_i(1)}-1},\quad 1\leq i\leq n.$$
Furthermore, $\mu= \infty$ (meaning that quantity in the argument of $v_p$ on the left-hand side of (\ref{eq:widemoment}) is zero for $n$ sufficiently large) if and only if $\rho(\nu)=0$.
\end{proposition}
\begin{proof}
Let $\rho(\nu)\in \Lambda_{R,m}^\mathrm{dig}$ be the horizontal measure associated to $\nu$ via the projection $\rho$. Applying the $p^m$-adic digit map we obtain a vertical measure $d_{\infty,m}(\rho(\nu))\in R\llbracket \Z_p\rrbracket$. Let $f_{\nu,m}(T)\in R\llbracket T\rrbracket $ be the power series associated to $d_{\infty,m}(\rho(\nu))$ via the Amice transform (Theorem \ref{thm:Amice}). Then by Corollary \ref{cor:lastdigit} we see that the left-hand side of (\ref{eq:widemoment}) is equal to 
\begin{align*}&v_p\left(\frac{\psi(1)-1}{\psi(p^{mn})-1} f_{\nu,m}(\zeta_{p^{m(n+1)}}-1)\right)\\
&= v_p\left(f_{\nu,m}(\zeta_{p^{m(n+1)}}-1)\right)+\tfrac{1}{p^{m(n+1)}-p^{m(n+1)-1}}-\tfrac{1}{p^m-p^{m-1}}.
\end{align*}
Assume that $\rho(\nu)\neq 0$. Then since the digit map is an $R$-isomorphism we conclude that $d_{\infty,m}(\rho(\nu))\in R\llbracket \Z_p \rrbracket$ is non-zero. Thus by the Weierstrass Preparation Theorem (Theorem \ref{thm:Weierstrassprep}) we can write
\begin{equation}\label{eq:Weierstrassprep}f_{\nu,m}(T)=\pi^{\mu} u(T) \left(T^{\lambda}+\sum_{i=1}^{\lambda} T^{\lambda-i}a_i \right) ,\end{equation}
where  $\pi$ is the uniformizer of $R$, $\mu=\mu(f_{\nu,m})$, $\lambda=\lambda(f_{\nu,m})$, $v_p(a_i)>0$ for $1\leq i \leq \lambda$ and $u(T)\in R\llbracket T\rrbracket^\times$. Now as $n\rightarrow \infty$ we see that 
$$v_p\left( (\zeta_{p^{m(n+1)}}-1)^{\lambda}\right)= \frac{\lambda}{p^{m(n+1)}-p^{m(n+1)-1}}\rightarrow 0.$$
This implies by (\ref{eq:Weierstrassprep}) that for $n$ large enough we have
$$v_p\left(f_{\nu,m}(\zeta_{p^{m(n+1)}}-1)\right)=\frac{\mu}{e}+\frac{\lambda}{p^{m(n+1)}-p^{m(n+1)-1}}, $$
since $u(\zeta_{p^{m(n+1)}}-1)$ is a $p$-adic unit. This yields the desired evaluation upon rearranging the terms appropriately. 
\end{proof}

\subsection{Proof of the main theorem}We are now in a position to prove our main theorem. 

\begin{proof}[Proof of Theorem \ref{thm:widemoment2}]Let $f_1,\ldots, f_M$ be holomorphic newforms of even weights $k_1,\ldots, k_M$. Applying Proposition \ref{prop:widemoment} to $\nu = \prod_{i = 1}^M\nu^{\pm}_{f_i,\mathcal{L},r}$ from (\ref{eq:horpadicL}), with $\rho:\Lambda_{R}^\mathrm{hor}\rightarrow \Lambda_{R,1}^\mathrm{dig}$ equal to the pushforward to the $p$-adic digit algebra defined by the canonical projections $\rho_n:\Z/p^{m_n}\rightarrow  \Z/p$, we get all the assertions once we note that 
\[\rho(\prod_{i = 1}^M\nu^{\pm}_{f_i,\mathcal{L},r})(\tilde{\chi}_1,\ldots,\tilde{\chi}_{n}) = \prod_{i = 1}^ML_f^{\ast}(\chi_1\cdots \chi_n,k_i/2)\]
by the interpolation property (\ref{eq:interpolationmeasure}), for $\chi_i\modulo \ell_i$ an order $p$ Dirichlet characters corresponding to $\tilde{\chi}_i: \Z/p\rightarrow \C_p^\times$ (i.e. $\tilde{\chi}_i(1)=\chi_i(b_i)$ where $b_i$ denotes the (chosen) generators of $(\Z/\ell_i)^\times$).
\end{proof}

\begin{remark}It is clear that Proposition \ref{prop:widemoment} can be applied to $\nu = \prod_{i = 1}^M\nu^{\pm}_{f_i,\mathcal{L},r}$ with $\rho : \Lambda_R^{\mathrm{hor}} \rightarrow \Lambda_{R,m}^{\mathrm{dig}}$ with any $m \in \Z_{\ge 1}$ to yield a generalization of Theorem \ref{thm:widemoment2}. For simplicity of exposition, we have only written out the statement for $m = 1$.
\end{remark}
\section{On the constants in the asymptotic formula}\label{sec:constant} The quantities $\mu$ and $\lambda$  appearing in (\ref{eq:widemoment2}) are in general very mysterious and might even depend on the ordering of the Taylor--Wiles primes. In this section we will show how under certain assumptions one can relate them to natural invariants of horizontal measure (see \cite[Definition 2.15]{KrizNordentoft}) and  determine them in terms of arithmetic invariants. 

\subsection{Kato--Kolyvagin derivatives}\label{sec:firstdigit} For the rest of this section, let $R$ be the ring of integers of a finite extension of $\Q_p$. Let $\nu\in \Lambda_{R}^\mathrm{dig}$ be an element of the digit algebra and denote by $f_{\nu}(T)\in R\llbracket T\rrbracket$ the associated power series obtained by applying the Amice transform (\ref{eq:amice}) to $d_{\infty,1}(\mu) \in R\llbracket \Z_p\rrbracket$.  Let $\pi_r:(\Z/p)^\N\twoheadrightarrow (\Z/p)^r$ denote the projection to the first $r$ factors and let $\nu_r:=\pi_r(\nu)\in R[(\Z/p)^r]$ denote the measure obtained by pushforward. For an integer $r\geq 0$, we define the \emph{$r^\mathrm{th}$ Kato--Kolyvagin derivative of $\nu$}  to be 
\begin{align}\label{eq:Dr}
 D^r \nu&:=\sum_{a_1=1}^{p}\cdots \sum_{a_r=1}^{p} \left(\prod_{i=1}^r a_i\right)\nu(\mathbf{1}_{\pi_r^{-1}\{(a_1,\ldots, a_r)\}})\\
 &=\sum_{a_1=1}^{p}\cdots \sum_{a_r=1}^{p} \left(\prod_{i=1}^r a_i\right)\nu_r([(a_1,\ldots, a_r)])\in R.  
 \end{align} 
These correspond to Kolyvagin derivatives considered in \cite{WZhang} in the setting of elliptic curves over $\mathbb{Q}$ considered over imaginary quadratic fields. When $r=0$ this is understood to mean $D^0\nu=\nu(\mathbf{1})$ where $\mathbf{1}$ denotes the trivial character. When $r = 1$ we will write $D^r = D$ for brevity. 

The main result of this section is a result relating the power series associated via the digit construction with the Kato--Kolyvagin derivative. For simplicity of exposition, we will henceforth only consider the case $r \leq  1$, although many of the results in the sequel admit generalizations to general $r$ by considering analogues of the digit map to power series algebras over $\Z_p$ in $r$ variables.
\begin{proposition}\label{prop:kolyvaginderiv}
Let $\nu\in \Lambda_R^\mathrm{dig}$ be a horizontal measure and denote by $f_{\nu}(T)\in R\llbracket T\rrbracket$ the associated power series obtained by applying the Amice transform (Theorem \ref{thm:Amice}) to $d_{\infty,1}(\mu) \in R\llbracket \Z_p\rrbracket$ (where $d_{\infty,1}$ is the digit map of Definition \ref{def:digitgen}). Then we have
\begin{align}
    f_{\nu}(0)= D^0\nu=\nu(\mathbf{1})
\end{align}
and the following congruences between elements of $R$, 
\begin{align}\label{eq:derivativecong} \frac{d}{dT}f_{\nu}(T)_{|T=0}&\equiv D\nu\modulo p. 
\end{align}
\end{proposition} 
\begin{proof}
By the definition of the Amice transform (Theorem \ref{thm:Amice}) we have
$$f_{\nu}(T)= \sum_{n = 0}^{\infty}d_{\infty,1}(\nu)\left(\binom{x}{n}\right)T^n ,$$
where $d_{\infty,1}: (\Z/p)^\N\rightarrow \Z_p$ denotes the digit map defined above. By plugging in $T=0$ we conclude the first equality.

Secondly, by taking the derivative in $T$ we  arrive at  
\begin{equation}\label{eq:partialderiv}\frac{d}{dT}f_{\nu}(T)_{|T=0}=d_{\infty,1}(\nu)(x),\end{equation} 
where $x:\Z_p\rightarrow \Z_p$ denotes the moment function $x\mapsto x$. This we can rewrite modulo $p$ as follows
\begin{align}\label{eq:momcongr} d_{\infty,1}(\nu)(x)
&\equiv \sum_{a=0}^{p-1}a\cdot d_{\infty,1}(\nu)\left(\mathbf{1}_{a+p\Z_p}\right)\\
& \equiv\sum_{a=0}^{p-1}a\cdot \nu\left(\mathbf{1}_{(\pi_1)^{-1}\{a\}}\right)\equiv D\nu \modulo p,\end{align}
which yields the desired statement.
\end{proof}
\begin{corollary}\label{cor:kolynonvan} Let $\nu\in \Lambda_R^\mathrm{dig}$ be a horizontal measure. Then:
\begin{enumerate}
    \item $D^0\nu=\nu(\mathbf{1})$ is a $p$-unit if and only if $\mu(f_\nu)=\lambda(f_\nu)=0$.
    \item $D^0\nu=\nu(\mathbf{1})$ is not a $p$-unit and $D\nu$ is a $p$-unit if and only if  $\mu(f_\nu)=0$ and $\lambda(f_\nu)=1$.
\end{enumerate}
\end{corollary}
\begin{proof}
(1) From Proposition \ref{prop:kolyvaginderiv} we conclude $\nu(\mathbf{1})$ is a $p$-unit exactly if $f_\nu(0)$ is a $p$-unit, which is equivalent to $f_\nu$ being invertible as an element of $R\llbracket T \rrbracket $. By Theorem \ref{thm:Weierstrassprep} this is equivalent to $\mu(f_\nu)=\lambda(f_\nu)=0$.

\noindent (2)  To prove the first implication, note that by Proposition \ref{prop:kolyvaginderiv} the assumption that $D\nu$ is a $p$-unit implies that $\frac{d}{dT}f_{\nu}(T)_{|T=0}$ is a $p$-unit. Taking a derivative in Theorem \ref{thm:Weierstrassprep}, we conclude that $\mu(f_\nu)=0$ and $\lambda(f_\nu)\leq 1$. Since $D^0\nu$ is not a $p$-unit we conclude that $\lambda(f_\nu)=1$. The other implication follows easily by a similar argument.
\end{proof}
We note that under the assumptions of  Corollary \ref{cor:kolynonvan}, it holds that $\mu(f_\nu)=\mu(\nu)$ and $\lambda(f_\nu)=\lambda(\nu)$ where $\mu(\nu),\lambda(\nu)$ denotes the \emph{horizontal Iwasawa invariants} of $\nu$ as defined in \cite[Definition 2.15]{KrizNordentoft}. In general, such a relation does not seem to hold.
\subsubsection{Arithmetic instances} In this final section we will explain how to apply Corollary \ref{cor:kolynonvan} to obtain the values of $\mu,\lambda$ as claimed in Remark \ref{rem:Kurihara}. The hypothesis of Corollary \ref{cor:kolynonvan} that $D^r\nu$ is a unit for $r\leq 1$  can be verified in many arithmetic situations via known cases of Kolyvagin's conjecture. We refer to \cite[Section 5.3]{KrizNordentoft} for more details on these and related matters.  \begin{enumerate}
    \item Assume that  $f_1,\ldots, f_m$ are holomorphic newform of even weights $k_i$  such that $p$ is $(f_1,\ldots, f_m)$-good, $L(f_i,k_i/2)\neq 0$ and the normalized $L$-values $L(f_i,k_i/2)/\Omega_{f_i}^+$ are all $p$-units. Note that the last condition holds for $p$ sufficiently large given the non-vanishing condition. Then by definition $\nu(\mathbf{1})$ is a $p$-unit where $\nu$ is the horizontal measure given by the product of the (plus) horizontal $p$-adic $L$-functions attached to $f_1,\ldots,f_m$ (using the joint Taylor--Wiles primes). Thus in this case we conclude from  Corollary \ref{cor:kolynonvan} that $\mu=\lambda=0$ in the asymptotic formula (\ref{eq:widemoment2}).
    \item Let $p > 2$ be an odd prime and let $E/\Q$ be an elliptic curve with positive analytic rank and with good ordinary reduction at $p$ such that $\mathrm{Sel}_{p^{\infty}}(E/\Q)$ has $\Z_p$-corank 1. Then by the main result of \cite{BurungaleCastellaGrossiSkinner} on the cyclotomic Kolyvagin conjecture, one can find a (non-Taylor--Wiles) prime $q_1\equiv 1\modulo p$ such that the congruence (\ref{eq:Kolderiv}) is satisfied. It now follows from a quick calculation (see \cite[Corollary 5.16]{KrizNordentoft}) that $D\nu$ is a $p$-unit  when $\nu$ is the horizontal $p$-adic $L$-function attached to $E$ and a sequence of primes $\ell_1,\ell_2,\ldots $ such that $\ell_1=q_1$ and $\ell_n$ is Taylor--Wiles for $n>1$ (assuming here implicitly that $p$ is $f_E$-good so that the horizontal $p$-adic $L$-function exists). Thus by Corollary \ref{cor:kolynonvan} we conclude that $\mu=0$ and $\lambda=1$ in this case. 
\end{enumerate}
 We conclude in both cases that the $\lambda$-invariant is closely related to the rank of the elliptic curve.

\bibliography{bibtexmain}
\bibliographystyle{amsplain}

\end{document}